\documentclass[12pt, reqno,a4paper]{amsart}
\usepackage{amsmath}
\usepackage{amssymb}
\usepackage{amsthm}
\usepackage{graphicx}
\usepackage{psfrag}
\usepackage{hyperref}
\usepackage{hyperref, enumitem}
\usepackage{enumitem, hyperref}
\usepackage{color,microtype}

\usepackage{tikz}
\usetikzlibrary{arrows}
\usetikzlibrary{decorations.pathreplacing}

\def\p{\mathbb{P}}
\def\sm{\setminus}

\makeatletter
\def\namedlabel#1#2{\begingroup
    #2%
    \def\@currentlabel{#2}%
    \phantomsection\label{#1}\endgroup
}
\makeatother

\allowdisplaybreaks

\theoremstyle{plain}
\newtheorem{theorem}{Theorem}[section]
\newtheorem{lemma}[theorem]{Lemma}
\newtheorem{proposition}[theorem]{Proposition}
\newtheorem{corollary}[theorem]{Corollary}
\theoremstyle{remark}
\newtheorem{remark}[theorem]{Remark}
\newtheorem{definition}[theorem]{Definition}
\newtheorem{example}{Example}[section]
\numberwithin{equation}{section}

\newcommand{\XX}{\mathcal{X}}
\newcommand{\U}{\mathcal{U}}

\newcommand{\eps}{\varepsilon}

\newcommand{\dist}{d}

\title{On distributional limit laws for recurrence}

\author[M.P.~Holland]{Mark Holland}
\address{Mathematics (CEMPS)\\
Harrison Building (327)\\
North Park Road\\
Exeter, EX4 4QF\\
England
} 
\email{m.p.holland@exeter.ac.uk.}
\urladdr{https://empslocal.ex.ac.uk/people/staff/mph204/}    

\author[M.~Todd]{Mike Todd}
\address{Mathematical Institute\\
University of St Andrews\\
North Haugh\\
St Andrews\\
KY16 9SS\\
Scotland} 
\email{m.todd@st-andrews.ac.uk}
\urladdr{https://mtoddm.github.io}    

\begin{document}

\date{\today}

\begin{abstract}
For a probability measure preserving dynamical system $(\XX,f,\mu)$, the Poincar\'e Recurrence
Theorem asserts that $\mu$-almost every orbit is recurrent with respect to its initial condition.
This motivates study of the statistics of the process $X_n(x)=\dist(f^n(x),x))$, and real-valued functions thereof. For a wide class of non-uniformly expanding dynamical systems, we show that 
the time-$n$ counting process $R_n(x)$ associated to the number recurrences below a certain 
radii sequence $r_n(\tau)$ follows an \emph{averaged} Poisson distribution $G(\tau)$. Furthermore, we obtain quantitative results on almost sure rates for the recurrence statistics of the process $X_n$.
\end{abstract}

\thanks{ MT was partially supported by the FCT (Fundação para a Ciência e a Tecnologia) project 2022.07167.PTDC.  MH acknowledges University of St Andrews for hospitality. Both authors gratefully acknowledge LMS Scheme 4 funding. The authors thank T. Persson for useful comments.}



\maketitle

\section{Introduction}
We consider a measure preserving dynamical system $(\XX,f,\mu)$, where $(\XX, d)$ is a metric space, and $f:\XX\to\XX$ is
a discrete time map preserving the (Borel) measure $\mu$. When $\mu(\XX)=1$, we know from the Poincar\'e Recurrence Theorem that for $\mu$-almost every $x\in\XX$, there exists a subsequence $(n_k)$ such
that $\dist(f^{n_k}(x),x)\to 0$. A fundamental problem in quantitative recurrence theory is to establish
limit laws governing the statistics of the process $X_{n}(x)=\dist(f^n(x),x)$, such as those
governing the frequency of close returns of $f^n(x)$ to $x$. We consider distributional limit
laws associated to the number of visits to sets of the form $\{X_n(x)<r_n\}$, with $r_n\to 0$.
  
Given a sequence $(r_n)$ we consider the sequence of balls $B_n(x)$ with $B_n(x)=B(x,r_n)$. To fix notation,
let
\begin{equation*}\label{eq:counting_rec}
R_n(r,x):=\#\{j\in[1,n]:\,\dist(x,f^j(x))\leq r\}.
\end{equation*}
Equivalently,
$$R_n(r_n,x)=\sum_{j=1}^{n}\mathbf{1}_{B_n(x)}(f^jx),$$
where $\mathbf{1}_A(x)$ denotes the indicator function for the set $A$. Thus $R_n(r_n,x)$ is
a counting function of close recurrence relative to the sequence $(r_n)$.
Related to recurrence is that of \emph{hitting}.
Let $A\subset X$ be a measurable set, and consider the counting process defined by
\begin{equation*}\label{eq:counting_rv}
N_n(A,x):=\#\{j\in[1,n]:\,f^j(x)\in A\}=\sum_{j=1}^{n}\mathbf{1}_{A}(f^jx).
\end{equation*}
In the case $A$ does not depend on $x$ (or it's orbit), then $N_n(A,x)$ is a hitting-time
counting function relative to $A$. For example, $A$ could be a ball $B(\zeta,r)$ centred
at some $\zeta\in\mathcal{X}$, and $r$ could depend on $n$ with $r_n\to 0$.

In this paper our first main result is Theorem~\ref{thm:poisson_like} where we establish distributional limit laws for the process $R_n(r_n,x)$ and contrast it
with those limit laws associated to the process $N_n(A,x)$ with $A=B(\zeta,r_n)$, and $\zeta$ a generic
point in $\mathcal{X}$.
To obtain distributional limit laws, we seek a non-degenerate function $G(\tau)$, and a sequence
$r_n(\tau)$ such that
\begin{equation*}\label{eq:counting_dist}
R_n(r_n(\tau),x)\to G(\tau),\quad(n\to\infty),
\end{equation*}
with convergence in distribution. Here $\tau>0$ is a real number. Since $R_n$ is a counting random variable, we are also interested in the functional forms $G(\tau)\equiv G(\tau,k)$ in the cases 
$\lim_{n\to\infty}\mu\{x:R_n(r_n(\tau),x)=k\}$, for $k\in\mathbb{N}$. 

\medskip

To guide us on the functional form for $G$, consider an i.i.d processes $\hat{X}_n$ with 
distribution function $F_{\hat{X}}(u):=P(\hat{X}\leq u)$. Suppose that for given $\tau>0$ we choose $u_n(\tau)$
so that $n(1-F_{\hat{X}}(u_n))\to\tau$ as $n\to\infty$. Then
\begin{equation*}\label{eq:standard_poisson}
\p(\#\{j\leq n:\,\hat{X}_j>u_n\}=k)\to\frac{\tau^k e^{-\tau}}{k!},
\end{equation*}
which coincides with the Poisson distribution. For certain dynamical processes of the form
$Y_n=\phi(f^{n}(x))$, where $\phi:\XX\to\mathbb{R}$ is a real-valued observable, then corresponding
Poisson laws are established, see \cite{FFT1}. 
The form of the observable is usually important in obtaining the standard Poisson limit law, 
and usually we require $\phi(x)$ is \emph{distance-like} so that $\phi(x)=\psi(\dist(x,\zeta))$
for some monotone decreasing function $\psi:[0,\infty)\to[0,\infty]$. More importantly, the reference
point $\zeta$ is usually chosen independently of the process $(Y_n)_{n\geq 0}$. In our case,
as relevant for recurrence analysis, the point $\zeta$ is actually the initial orbit value. This significantly influences the limit laws that arise relative to the case where $\zeta$ is chosen
independent of the orbit.

Our second main result is Theorem~\ref{thm:almost-sure}. This result gives almost sure bounds
(both lower and upper) on $\min_{j\leq n}d(f^jx,x)$. Such results complement the distributional limit laws obtained for $R_n$, and are based on a more refined analysis of the event $\{R_n(r_n,x)=0\}$.
These results produce a further understanding of quantitative Poincar\'e recurrence statistics compared to works of Boshernitzan \cite{Boshernitzan}, and more recent works such as \cite{KKP2,Persson}.

This paper is organised as follows. In Section \ref{sec.main-results} we introduce the main assumptions
on the system $(\XX,f,\mu)$. These assumptions are based upon the mixing properties (such as having exponential decay of correlations), regularity of the invariant density (e.g., of bounded variation type), and quantitative estimates on the recurrence time statistics over short time scales. We then state the main results, namely Theorem~\ref{thm:poisson_like} concerning
a distributional limit law for $R_n$, and then Theorem~\ref{thm:almost-sure} regarding an almost sure limit law for $R_n$. In Section \ref{sec.asym-hyp-rec} we show through specific case studies that the conclusions of the results obtained in Section \ref{sec.main-results} can also apply to systems with weaker assumptions on their mixing rates, and/or assumptions on the invariant density. This captures systems with polynomial decay of correlations, and/or systems with unbounded invariant density. The main result here
is Theorem~\ref{thm.asym-hyp-rec}. 
In Sections~\ref{sec.pf.thm:poisson_like}, \ref{sec.pf.thm:almost-sure} and \ref{pf:thm.asym-hyp-rec} we prove Theorems~\ref{thm:poisson_like}, \ref{thm:almost-sure} and \ref{thm.asym-hyp-rec} respectively.  Finally in the appendix we give some extensions of Theorem~\ref{thm:almost-sure}.

\section{Uniformly expanding interval maps}\label{sec.main-results}

We consider a measure preserving system $(\XX,f,\mu)$, with $\mu$ an ergodic
measure. The systems we consider include those that can be
modelled by a Young tower \cite{Young}. However we do not need the precise Young tower
construction in our analysis. We focus on interval maps rather than general hyperbolic systems.
Applications include piecewise expanding interval maps \cite{LasotaYorke}, quadratic maps at Benedicks-Carleson parameters \cite{Young_quadratic}, and maps with neutral fixed points (intermittent maps) \cite{FFTV}
 
We introduce the following definitions and assumptions. We begin with the definition of decay of correlations. 
\begin{definition}\label{def.dc}
  We say that $(\XX,f,\mu)$ has decay of correlations in (Banach
  spaces) $\mathcal{B}_1$ versus $\mathcal{B}_2$, with rate
  function $\Theta(j)\to 0$ if for all
  $\varphi_1\in\mathcal{B}_1$ and $\varphi_2\in\mathcal{B}_2$ we
  have
  \[
  \mathcal{C}_j(\varphi_1,\varphi_2,\mu):= \biggl|\int \varphi_1
  \cdot \varphi_2\circ f^j \, \mathrm{d} \mu -\int \varphi_1 \,
  \mathrm{d} \mu \int \varphi_2 \, \mathrm{d} \mu \biggr| \leq
  \Theta(j) \|\varphi_1\|_{\mathcal{B}_1}
  \|\varphi_2\|_{\mathcal{B}_2},
  \]
  where $\|\cdot\|_{\mathcal{B}_i}$ denote the corresponding
  norms on the Banach spaces.
\end{definition}

In particular, we consider the $L^p$ and $\mathrm{BV}$ norms of
functions $\varphi \colon \XX \subset \mathbb{R} \to
\mathbb{R}$ defined respectively by
\begin{align*}
  \|\varphi \|_p & := \left(\int |\varphi|^{p} \, \mathrm{d}\mu\right)^{\frac{1}{p}},
  \\ \|\varphi\|_{\mathrm{BV}} & := \mathrm{var}(\varphi) +
  \sup(|\varphi|),
\end{align*}
where $\mathrm{var}(\varphi)$ denotes the total variation of
$\varphi$. In particular, for an interval $J\subset\XX$,
$$\mathrm{var}_J(\varphi):=\sup_{\mathcal{P}}\sum_{i=1}^{m}|\varphi(x_i)-\varphi(x_{i-1})|,$$
where $\mathcal{P}$ is the partition of $J$ into intervals $[x_i,x_{i+1}]$, with $\{x_0,x_m\}=\partial J$.
The supremum is then taken over all such partitions. When $J=\mathcal{X}$, we just denote this \emph{variation}
by $\mathrm{var}(\varphi)$. 
Functions $\varphi \colon \XX \subset \mathbb{R} \to
\mathbb{R}$ with $\|\varphi\|_{\mathrm{BV}}<\infty$ are called
functions of \emph{bounded variation}.  
Throughout this paper we assume that our measure $\mu$ is ergodic and absolutely continuous with respect to Lebesgue: an \emph{acip}.

The first main assumption is the following.
\begin{enumerate}
\item[(A1)] We assume that
  $(\XX,f,\mu)$ has exponential decay of correlations in Banach
  spaces $\mathcal{B}_1=\mathrm{BV}$ versus
  $\mathcal{B}_2=L^{1}(\mu)$.
\end{enumerate}
We remark that within (A1), we could also have taken $\mathcal{B}_2=L^p(\mu)$ for some $p\in(1,\infty]$.
The proofs of the main results do not rely heavily on the value of $p$. Moreover the case $p=\infty$ can be 
replaced by assuming instead decay for some $p>1$, see \cite[Section 15]{HKKP}. For certain dynamical systems, decay of correlations may only be known for $\mathcal{B}_1$ the space of H\"older functions (rather
than $\mathrm{BV}$). As discussed in Section~\ref{sec.pf.thm:poisson_like}, our results also apply in this case by minor modification of the proofs. 

In addition to decay of correlations we make assumptions on the recurrence statistics of orbits over early time windows. We define the following sets:
\begin{align*}
E_{j}(r_n) &:=\{x\in\XX:\,\dist(x,f^jx)\leq r_n\}\\
\tilde{E}_g(r_n) &:=\bigcup_{j=1}^{g(n)}E_j(r_n),
\end{align*} 
where $(r_n)$ is sequence (that does not depend on $x$), and $g(n)$ a monotone increasing function with
$g(n)\to\infty$.
We make the following assumption. In the following $a_n\approx b_n$ means there is a uniform constant
$c>1$ such that $a_n/b_n\in[c^{-1},c]$ (as $n\to\infty$).
\begin{enumerate}
	\item[(A2)] Suppose that $r_n\approx n^{-a}$ for some $a>0$. There exists $\beta_0\in(0,1)$, 
	and $C>0$ such that for all $j\leq (\log n)^2$ we have
	\begin{equation*}	\label{eq:A1}
	\mu(E_{j}(r_n))\leq Cr_{n}^{\beta_0}.
	\end{equation*}
	with constant $C$ independent of $j$.
\end{enumerate}
Consequently, this implies that for rate function $g(n)=(\log n)^2$, there exists $\beta_1\in(0,1)$ 
such that \begin{equation*}
\mu(\tilde{E}_{g}(r_n))\leq Cr_{n}^{\beta_1}.
\end{equation*}
\begin{definition}\label{def.hyp-rec}
We say that a measure preserving dynamical system $(\XX,f,\mu)$ is \emph{hyperbolic recurrent} if
it satisfies Assumptions (A1) and (A2), and has an ergodic measure $\mu$ with
invariant density $\rho\in\mathrm{BV}$.
\end{definition}
Examples satisfying Definition \ref{def.hyp-rec} include the Gauss map 
$f(x)=\lfloor\frac{1}{x}\rfloor\mod 1$, beta-transformations $f(x)=\beta x\mod 1$,
$(\beta >1)$, Rychlik maps, and Gibbs Markov maps \cite{HNT1} (in all these cases $\mu$ is an acip). 
In particular, regarding verification of (A2) see \cite[Section 13]{HKKP}, and \cite{Letal}. 
Our first main result is the following.
\begin{theorem}\label{thm:poisson_like}
Suppose that $(\XX,f,\mu)$ is a measure preserving dynamical system which is hyperbolic
recurrent in the sense of Definition \ref{def.hyp-rec}. Consider the process
$$X_n=\dist(f^nx,x),\quad
R_n(r_n,x):=\#\{1\leq j\leq n:\,X_j\leq r_n\},$$
with $r_n=\tau/2n$ and $\tau>0$. Then for all $k\geq 0$,
\begin{equation}\label{eq:poisson_like}
\lim_{n\to\infty}\mu(x\in\XX:\,R_n(r_n,x)=k)=\int_{\mathcal{X}}\frac{\tau^k\rho(x)^{k+1}}{k!}
e^{-\rho(x)\tau}\,dx.
\end{equation}
\end{theorem}
We now make several remarks. 
\begin{remark} 
The limit law appearing here is not Poisson, but Poisson-like in the sense that the Poisson parameter 
$\tau$ is weighted against the density $\rho(x)$ and then averaged over $\XX$ w.r.t.\ $\mu$. In the case of hitting time statistics where instead the process is given by $\widetilde{X}_n(x)=
\dist(f^nx,\zeta)$, (thus defined relative to a reference point $\zeta\in \XX$), the 
parameter $\tau$ is governed by the density at $\zeta$. Moreover existence of a Poisson limit-law depends on the recurrence properties of $\zeta$, for example $\zeta$ periodic versus non-periodic. In our case, there is no fixed reference point $\zeta$, since our time series is based on recurrence (close returns) to the initial orbit value. 
\end{remark}
\begin{remark}\label{rmk.collet}
For the special case $k=0$,
Collet \cite{Collet} obtained a similar limit law for certain one-dimensional non-uniformly expanding dynamical systems, such as those modelled by Young towers with exponential tails. In fact Collet 
considered the process $X_n=-\log\dist(f^nx,x)$ with modified scaling sequence 
$r_n(u)=u+\log n$. The corresponding limit law arising is that given in equation
\eqref{eq:poisson_like} (in the case $k=0$), but replacing $\tau$ by $2e^{-u}$.  Note that \cite[Proposition 3.5]{PenSau10} gives the same result in the context of billiard maps.
In Section \ref{sec.psi} we comment further on the limit laws that arise when looking at functions of
$d(f^nx,x)$.
\end{remark}
\begin{remark}
More recently, a result similar to the limit law appearing in equation \ref{eq:poisson_like} was obtained in \cite[Theorem 5.3]{DolFayLiu22}. They make the assumption of $r$-fold exponential mixing for all $r$, see \cite[Definition 3.1]{DolFayLiu22}. Our approach has the advantage of obtaining error terms in the convergence to distribution. This allows us to obtain almost sure results on the quantitative recurrence
statistics. See Section \ref{sec.almost-sure} ahead. We also obtain corresponding limit laws for systems having weaker (sub-exponential) assumptions on the rate of mixing, see Section \ref{sec.asym-hyp-rec}.
\end{remark}
As an immediately corollary, we have the following standard Poisson-law for the doubling map.
\begin{corollary}\label{cor:poisson_Leb}
Consider the doubling map $f(x)=2x\mod 1$ (preserving Lebesgue measure). Then along the
sequence $r_n=\tau/2n$ we have
\begin{equation*}\label{eq:poisson_like2}
\lim_{n\to\infty}\mathrm{Leb}\left(x\in\XX:\,R_n(r_n,x)=k\right)=\frac{\tau^k}{k!}
e^{-\tau}.
\end{equation*}
\end{corollary}
 
\subsection{On almost sure behaviour of recurrence statistics.}\label{sec.almost-sure}
In this section we consider almost sure behaviour of the recurrence statistics, that is the determination
of sequences $u_n,v_n$ such that
$$\mu\left(x\in\XX:\,u_n<\min_{k\leq n}\dist(f^kx,x)<v_n,\;\mathrm{eventually}\right)=1.$$
For the related hitting time process $Y_n(x)=\psi(\dist(f^nx,\zeta))$, with $\zeta$ independent
of the orbit of $x$, and $\psi$ a monotone decreasing function, see \cite{GNO,HKKP,HNT2} concerning
almost sure results. For recurrence, there is much recent interest on obtaining almost sure bounds on 
$\min_{k\leq n} X_n(x)$ and solving related questions on determining the probability of the event 
$\{x\in\XX:d(f^nx,x)\in[0,r_n]\,i.o.\}$ when $r_n\to 0$. (Here $\mathrm{i.o.}$ means
infinitely often). We refer the reader to Boshernitzan 
\cite{Boshernitzan}, and more recently \cite{KKP2,Persson}. The following is our second main result.
\begin{theorem}\label{thm:almost-sure}
Suppose that $(\XX,f,\mu)$ satisfies the hypotheses of Theorem \ref{thm:poisson_like}.
Then we have the following cases.
\begin{enumerate}
\item[(a)] (Almost sure upper bounds). Consider the following alternatives:
\begin{enumerate}
\item[(i)] The density $\rho$ is strictly positive, and
$(r_n)$ is a sequence satisfying $r_n>\frac{c\log\log n}{n}$, with
$2c(\inf_{x\in\XX}\rho(x))>1$.
\item[(ii)] The density $\rho(x)$ is uniformly H\"older continuous, and
$(r_n)$ is a sequence satisfying
$r_n>\frac{(\log n)^{\gamma}}{2n}$, with $\gamma>\gamma_0$, and
\begin{equation}\label{eq:asure-sum}
\sum_{k\geq 1}\int_{\XX}\rho(x)e^{-\epsilon k^{\gamma_0}\rho(x)}\,dx<\infty,
\end{equation}
valid for all $\epsilon>0$.
\end{enumerate}
Then in either case (i) or (ii) we have
$$\mu\left(x\in\XX:\,\min_{k\leq n}\dist(f^kx,x)<r_n\;\mathrm{eventually}\right)=1.$$
\item[(b)] (Almost sure lower bound). Suppose the sequence $(r_n)$ satisfies
$\sum_{n=1}^{\infty}\int_{\XX}\mu(B(x,r_n))\,d\mu<\infty.$
Then 
$$\mu\left(x\in\XX:\,\min_{k\leq n}\dist(f^kx,x)>r_n\;\mathrm{eventually}\right)=1.$$
\end{enumerate}
\end{theorem}
We make several remarks. For case (a), we can compare this result to \cite[Theorem E]{KKP2}, where they require $r_n>1/n^\gamma$ for $\gamma<1$. They also consider measures that are not necessarily acips,
and consider scenarios where $r_n$ depends on $x$. We show that case (b) follows from \cite[Theorem C]{KKP2}. To prove Theorem \ref{thm:almost-sure}, we use Theorem \ref{thm:poisson_like} together with an error term. This allows us to then use a First Borel Cantelli Lemma approach to establish case (a). The details are given in Section \ref{sec.pf.thm:almost-sure}. In particular proof of case (b), via \cite[Theorem C]{KKP2} is also based upon a First Borel Cantelli Lemma approach.

At first sight, equation \eqref{eq:asure-sum} looks mysterious. Analogous
criteria arise in \cite[Section 10]{HKKP}. In particular, case (a)(i) implies validity of \eqref{eq:asure-sum}. In (a)(ii), the assumption on the regularity of the density can be relaxed to piecewise H\"older, assuming a finite number of jump discontinuity points. This is discussed in the appendix. Furthermore, case (a)(ii) is most useful if the density $\rho(x)$ equals zero for some $x\in\XX$, as in
the following example.
\begin{example}\label{ex:zero}
Suppose $\XX=[0,1]$ and for $C>0$, $a>0$ we have $\rho(x)\in[C^{-1},C]x^{a}$. Thus
$\rho(x)\approx x^a$ as $x\to 0$. The integral in \eqref{eq:asure-sum} can be estimated
via
\begin{equation*}
\int_{0}^{1}\rho(x)e^{-\epsilon k^{\gamma}\rho(x)}\,dx \approx 
\int_{0}^{1} x^a e^{-\epsilon k^{\gamma}x^a}\,dx
\leq \frac{C_{a}}{(\epsilon k^{\gamma})^{1+\frac{1}{a}}}\int_{0}^{\infty} u^{\frac{1}{a}}e^{-u}\,du.
\end{equation*}
For the latter inequality, we made a change of variable, and $C_a$ is a constant. In particular the latter integral is also uniformly bounded. Hence, we see that \eqref{eq:asure-sum} applies provided
$\gamma>\frac{a}{a+1}$.
\end{example}
In particular, examples of dynamical systems exhibiting zeros in their invariant density are considered in \cite{CHMV}. However, these examples, `intermittent cusp maps' are non-uniformly expanding. 
We give details of these in Section \ref{sec.asym-hyp-rec}, showing how
 the conclusion of case (a)(ii) applies.  For case (a)(i), we can be explicit on the range of $c$ in certain examples, such as those in the following Corollary.
\begin{corollary}\label{cor:almost_sure_Leb}
Consider the doubling map $f(x)=2x\mod 1$ (preserving Lebesgue measure). Then case (a)(i)
of Theorem \ref{thm:almost-sure} applies along the sequence $r_n=c\log\log n/n$ with $c>1/2$.
\end{corollary} 

\subsection{On extreme statistics for recurrence}\label{sec.psi}
In Theorem \ref{thm:poisson_like}, we considered the process $X_n=\dist(f^nx,x)$.
A more general problem is to determine corresponding limit laws for the process
$X^{\psi}_n=\psi(\dist(f^nx,x))$, where $\psi:[0,\infty)\to\mathbb{R}$ is a monotone
decreasing function, thus taking maximum at $0$, i.e. when $f^n(x)=x$. 
As discussed in Remark \ref{rmk.collet}, the special case $\psi(z)=-\log z$ was considered in \cite{Collet}. 
Incorporation of $\psi$ is relevant in the study of \emph{Extreme Value Theory}, especially
for determining distribution limit laws governing the maximum process
\begin{equation}\label{eq.max-rec}
M^{\psi}_n(x):=\max_{1\leq k\leq n}\psi(\dist(f^kx,x)).
\end{equation}
In this case $r_n$ is defined relative a linear scaling sequence of the form $u/a_n+b_n$ via
$$\psi(r_n(u))=\frac{u}{a_n}+b_n,\quad a_n, b_n\in\mathbb{R}.$$ 
In the case of a hitting time maximum process
\begin{equation*}\label{eq.max-hit}
\widetilde{M}^{\psi}_n(x):=\max_{1\leq k\leq n}\psi(\dist(f^kx,\zeta)),
\end{equation*}
(defined relative to a reference point $\zeta\in \XX$), the scaling sequences $a_n,b_n$ and
limit law function $G(u)$ for $\lim_{n\to\infty}\mu\{\widetilde{M}^{\psi}_n(x)\leq u/a_n+b_n\}$ can be understood for a wide class of dynamical systems, see \cite{Letal}.
For the related recurrence maxima process given in \eqref{eq.max-rec} a full classification of the corresponding limits $G(u)$ (relative to a given $\psi$) is beyond the scope of this paper. 
We illustrate with the following examples.
\begin{example}
Suppose that $\psi(z)=-\log z$. We choose $a_n, b_n$ so that $\psi(r_n)$ matches $\tau/2n$, with
$\tau\equiv\tau(u)$. We take $a_n=1$ and $b_n=\log (2n)$ giving $\tau(u)=e^{-u}$. This example
is consistent with the result of \cite[Theorem 4.1]{Collet}. 
\end{example}
\begin{example}
Suppose that $\psi(z)=z^{-\alpha}$ with $\alpha>0$. In this case we obtain $\tau(u)=u^{-\frac{1}{\alpha}}$
in Theorem \ref{thm:poisson_like}
via sequences $a_n=(2n)^{-\alpha}$ and $b_n=0$. 
\end{example}
Thus for scaling sequences $r_n$ defined by $\psi(r_n)=u/a_n+b_n$, the canonical functional forms arising for $\tau(u)$ are consistent with the Type I-III distributions arising in extreme value theory (Gumbel, Fr\'echet and Weibull). In general, an explicit formula for $G(u)$ is not easily obtained, but corresponds
to an \emph{averaged} version of these extreme distribution types. We give an
example below where $G(u)$ can be made explicit. We also give an example in 
Section~\ref{subsec.asym-hyp-rec}.

\begin{example}\label{ex:beta}
Consider the transformation $f(x)=\beta x\mod 1$, with $\XX=[0,1]$ and $\beta^{-1}=\beta-1$ (with
$\beta\in(1,2)$). This map is uniformly expanding and Markov. The invariant density has explicit
formula
$$\rho(x)=\frac{\beta^3}{\beta^2+1}\mathbf{1}_{[0,\beta^{-1})}(x)
+\frac{\beta^2}{\beta^2+1}\mathbf{1}_{[\beta^{-1},1]}(x),$$
and is piecewise constant. We therefore obtain for each $k\geq 0$
\begin{multline*}
\mu(R_n(r_n,x)=k)\underset{n\to\infty}{\longrightarrow} \\ \frac{1}{k!}\left\{\frac{1}{\beta}\left(\frac{\beta^3\tau}{\beta^2+1}\right)^k
e^{-\tau\frac{\beta^3}{\beta^2+1}}
+\frac{1}{\beta^2}\left(\frac{\beta^2\tau}{\beta^2+1}\right)^ke^{-\tau\frac{\beta^2}{\beta^2+1}}
 \right\}.
\end{multline*}
Hence, once $\tau(u)$ is specified as in the previous examples (e.g. via choice of $\psi$ and appropriate scaling sequences $a_n,b_n$) then we obtain $G(u)$ accordingly.
\end{example}

\section{Dynamical systems with asymptotic hyperbolic recurrence}\label{sec.asym-hyp-rec}
In this section we consider dynamical systems that do not satisfy Definition \ref{def.hyp-rec}.
A scenario that can arise includes the case where (A1) fails to hold, such as the system having sub-exponential decay of correlations. Similarly, the assumption that the invariant density $\rho(x)$
is of bounded variation type is also restrictive. However, the sources that lead to (say) sub-exponential
decay of correlations or lack of regularity of $\rho$ maybe due to local phenomena, such as the map
having a neutral fixed point, and/or having a critical point. This allows us to sometimes approximate
those maps using dynamical systems that are hyperbolic recurrent. This leads to the following definition.

\begin{definition}
We say that $(\XX, f, \mu)$ is \emph{asymptotically hyperbolic} if for each $\eps>0$ we may remove some set $U_\eps= \cup_i U_{\eps, i}$ for some connected sets $U_{\eps, i}$ such that
\begin{enumerate}
\item[ (i)] $\mu(U_\eps) \le \eps$
\item[(ii)] for $Y_\eps:= \XX\setminus U_\eps$ and $\tau= \tau_\eps$ the first return time to $Y_\eps$, the system $(Y_\eps, F_\eps= f^{\tau_\eps}, \mu_\eps)$ is hyperbolic recurrent for $\mu_\eps = \mu|_{Y_{\eps}}/\mu(Y_\eps)$.  
\end{enumerate}
\end{definition}
We recall that $\mu$ is an acip, so that it is implicit in the definition of asymptotically hyperbolic
that $\mathbb{E}_{\mu_{\eps}}(\tau_{\eps})<\infty.$
Then we may apply Theorem~\ref{thm:poisson_like} to $(Y_\eps, F_\eps= f^{\tau_\eps}, \mu_\eps)$ to prove the following.

\begin{theorem}\label{thm.asym-hyp-rec}
If $(\XX, f, \mu)$ is asymptotically hyperbolic then the conclusions of Theorem~\ref{thm:poisson_like}
and Theorem~\ref{thm:almost-sure} hold for $(\XX, f, \mu)$.
\end{theorem}

We prove Theorem \ref{thm.asym-hyp-rec} in Section \ref{pf:thm.asym-hyp-rec}. In the next section we illustrate with examples for which Theorem \ref{thm.asym-hyp-rec} applies.

\subsection{Dynamical systems exhibiting asymptotic hyperbolic recurrence.}\label{subsec.asym-hyp-rec}
Given $\gamma\in (0, 1)$, define $f=f_\gamma:[0, 1]\to [0, 1]$ by 
\[f_\gamma(x) :=\begin{cases} x(1+2^\gamma x^\gamma) & \text{ if } x\in [0, 1/2),\\
2x-1 & \text{ if } x\in [1/2, 1].
\end{cases}
\]
This is a standard form of Manneville-Pomeau map and it has an acip $\mu=\mu_\gamma$.  Note that this is not hyperbolic recurrent, for example the density is not in BV.

Now define a Misiurewicz-Thurston unimodal map as follows.  Suppose that $f:[0, 1]\to [0, 1]$ is $C^2$,  has negative Schwarzian derivative and has a unique critical point $c$, which is moreover non-flat\footnote{In fact we can also handle systems with flat critical points as in \cite{Zwe04}, but we omit the details here.}.  Furthermore we assume that there is some $n\ge 2$ such that $f^n(c)$ is a repelling periodic point.  For details of these maps see for example \cite{DemTod23}, a basic example 
being $x\mapsto 4x(1-x)$. Again there is an acip $\mu$, but the system is not hyperbolic recurrent, the density is not in BV.

\begin{proposition}
If $([0, 1], f, \mu)$ is a Manneville-Pomeau map with $\mu$ an acip, or is a Misiurewicz-Thurston unimodal map, with $\mu$ an acip, then it is asymptotically hyperbolic.
\end{proposition}

\begin{proof}
In the Manneville-Pomeau case, we can ensure that $(Y_\eps, F_\eps, \mu_\eps)$ is Markov by removing some $[0, x_n)$ for $x_n$ such that $f^n(x_n)=1/2$ and $f^j(x_n) \in [0, 1/2]$ for $j=0, \ldots, n-1$.  Note that $\mu([0, x_n))\to 0$ as $n\to \infty$.  In this case the system $(Y_\eps, F_\eps, \mu_\eps)$ 
is Gibbs-Markov and is hence hyperbolic recurrent. 

In the Misiurewicz-Thurston case we may remove an $\eps$-neighbourhood of the orbit of $f(c)$.  This can be done in such a way that the resulting system $(Y_\eps, F_\eps, \mu_\eps)$ is Markov by exploiting the Markov structure of the original system, see for example \cite[Section 2.1]{DemTod23}.  Again this system is then Gibbs-Markov and hence hyperbolic recurrent.
\end{proof} 

\subsubsection*{The intermittent cusp map.}
We consider the dynamical system $(\XX,f,\mu)$ given by $f(x)=1-2\sqrt{|x|}$ for $\XX=[-1,1]$.
This map has a neutral fixed point at $x=-1$ (quadratic tangent with diagonal), and a derivative
singularity at $x=0$. In particular we have $|f'(x)|=|x|^{-1/2}$. However, the statistical properties 
are well understood and this map fits the definition of being asymptotically hyperbolic, see \cite{CHMV}. The relevant inducing set can be taken to be $Y_{\epsilon}=[-1+\epsilon, 0)\cup(0,1-\epsilon]$. The map has an ergodic invariant (probability) density function given by $\rho(x)=\frac{1}{2}(1-x)$. Hence,
relative (say) to the Manneville-Pomeau map we can make the limit distribution function 
for $\mu(R_n(r_n,x)=k)$ explicit, and also determine more precisely the almost sure rates governing the growth of 
$\min_{j\leq n}X_j(x)$.

To determine the limit distribution, we can apply Theorem~\ref{thm.asym-hyp-rec} to deduce that
along $r_n=\tau/2n$ we obtain
\begin{equation*}
\lim_{n\to\infty}\mu(R(x,r_n)=k)=\int_{-1}^{1}\frac{\tau^k}{k!}\left(\frac{1-x}{2}\right)^{k+1}
e^{-\frac{\tau}{2}(1-x)}\,dx.
\end{equation*} 
The right-hand integral can be evaluated easily to give the limit distribution $G(\tau)$ given by
\begin{equation*}
G(\tau)=2(k+1)\left\{\frac{1}{\tau^2}-e^{-\tau}\sum_{j=0}^{k+1}\frac{\tau^{k-j-1}}{(k+1-j)!}\right\}.
\end{equation*}
In particular when $k=0$, we obtain
\begin{equation*}
G(\tau)=\frac{2}{\tau^2}-2e^{-\tau}(\tau^{-1}+\tau^{-2}).
\end{equation*}
Notice that for large $\tau$, the limit distribution is heavy tailed (power law) in $\tau$, rather than of exponential type. We can contrast this limit distribution to that of Example \ref{ex:beta}. The mechanism
for the heavy tail in this example is due to the density $\rho(x)$ taking a value of zero (at $x=1$). 

For the almost sure bounds on $\min_{j\leq n}X_j(x)$, we apply case (a)(ii)
of Theorem \ref{thm:almost-sure} by taking $\gamma_0=1/2$ in the summability criteria of equation 
\eqref{eq:asure-sum}. This gives for all $\eta>0$,
$$\mu\left(\min_{j\leq n}X_{j}(x)<\frac{(\log n)^{\frac{1}{2}+\eta}}{n}\;\textrm{eventually}\right)=1.$$
For the lower bound, we can apply the summability criteria of case (b). Since $\rho(x)$ is bounded,
an almost sure lower bound is given by a sequence $r_n$ with $\sum_{n}r_n<\infty$.
However, if we follow the proof of Theorem~\ref{thm.asym-hyp-rec}, then
given $\epsilon>0$ and inducing set $Y_{\epsilon}$ we can find $c_0\equiv c_{\epsilon}$ so that 
case (a)(i) holds for $\mu$-almost every $x\in Y_{\epsilon}$ (along the sequence $r_n=c\log\log n/n$ with
$c>c_0$). In particular as $\epsilon\to 0$, $\mu(Y_{\eps})\to 1$, but the lower bound constant $c_0$ then tends to infinity.

\section{Proof of Theorem \ref{thm:poisson_like}}\label{sec.pf.thm:poisson_like}
The idea of proof is to first remove dependence on the initial condition $x$. This will be achieved using
decay of correlations. The second step of the proof then involves study of a hitting counting function
of the form $N_n(\Delta(\zeta_k),x)$, where $\Delta(\zeta_k)$ is a small (shrinking) interval centred on 
$\zeta_k$. The points $\zeta_k$ are
uniformly spaced sequence in $\XX$. As $\Delta$ shrinks (with increasing $n$), the number of such points 
$\zeta_k$ then increases. We are then led to average $N_n(\Delta(\zeta_k),x)$ over $\zeta_k$. This
averaging explains the integral arising in equation \eqref{eq:poisson_like}.
  
We remark that the case $\{R_n=0\}$ matches the conclusion
of \cite[Theorem 4.1]{Collet} in the case $\psi(z)=-\log z$. There, 
an \emph{averaged} exponential law distribution is obtained. For the case
$\{N^{\psi}_n=k\}$ a best approach here seems a variant
of the \emph{Chen-Stein method} as applied in \cite[Theorem 2.1]{CC}.  Compared to the aforementioned references, the setting here is simpler since our measure admits
a density $\rho(x)$ which is bounded. Moreover, decay of correlations
estimates for relevant indicator functions (appearing in the proofs) 
have bounded $\mathrm{BV}$ and $L^1$ norms. This avoids
the need to approximate indicator functions by Lipschitz (or H\"older) functions. 
Our proofs would still go through in these cases, but with additional minor modifications.

\subsubsection*{Notation} Regarding notation, for functions $g(x), h(x)$ defined either over
$\mathbb{R}$ or $\mathbb{N}$, we use the convention $f\lesssim g$
if there exists a constant $C>0$ such that $f(x)\leq Cg(x)$ as $x\to\infty$. Equivalently
we sometimes use the notation $f(x)=O(g(x))$.  We say that $f\approx g$ if there exists a constant
$C>0$ such that $C^{-1}\leq f(x)/g(x)\leq C$ for all sufficiently large $x$.
Unless stated otherwise, we use $C_1,C_2,\ldots$ to denote generic constants which do not depend on $x$. 

\subsection{Removing dependence on the initial condition} 
We proceed along similar lines to the proof of \cite[Theorem 4.1]{Collet}. Suppose $\mathcal{X}$
is the interval $[0,1]$, and set up a partition $\mathcal{U}_m$ of $\XX$ into equally sized intervals
$\Delta_k\equiv\Delta(\zeta_k)$, with $1\leq k\leq m$, and $\zeta_k$ the centre of $\Delta_k$. Thus each interval $\Delta_k$ has size $1/m$. We let $\delta=\frac{1}{2}|\Delta|$, where $\Delta$ is any
element of $\mathcal{U}_m$. We let $\zeta$ denote the centre of $\Delta$. Given $\Delta$, and $r\ge\delta$
we introduce sets $\Delta^{\pm}$ and $\Delta_0$ given by:
\begin{align*}
\Delta^{\pm} &:=\left\{x\in\XX:\,d(x,\Delta)\leq r\pm\delta\right\},\\
\Delta_0 &:= \left\{x\in\XX:\,d(x,\Delta)\leq r\right\},
\end{align*}  
see Figure~\ref{fig:Deltas} for a schematic picture of this.

\begin{figure}[h]
\begin{tikzpicture}[thick, scale=0.5]

\draw[ - ] (0,5) -- (20,5);
\draw[ thick, - ] (0,4.7) -- (0,5);
\draw[ thick, - ] (20,4.7) -- (20,5);
\draw (20,5) node[below] {{\small $\Delta^+$}};
\draw[ thick, - ] (10,2.5) -- (10,5);
\draw (10,2) node[below] {{\small $\zeta$}};

\draw[ - ] (1.5,4.3) -- (18.5,4.3);
\draw[ thick, - ] (1.5,4) -- (1.5,4.3);
\draw[ thick, - ] (18.5,4) -- (18.5,4.3);
\draw (18.5,4.2) node[below] {{\small $\Delta_0$}};

\draw[ - ] (3,3.6) -- (17,3.6);
\draw[ thick, - ] (3,3.3) -- (3,3.6);
\draw[ thick, - ] (17,3.3) -- (17,3.6);
\draw (17,3.6) node[below] {{\small $\Delta^-$}};

\draw[ - ] (8.5,2.9) -- (11.5,2.9);
\draw[ thick, - ] (8.5,2.6) -- (8.5,2.9);
\draw[ thick, - ] (11.5,2.6) -- (11.5,2.9);
\draw (11.5,2.6) node[below] {{\small $\Delta$}};

\draw[red, <-> ] (8.5,5.3) -- (10,5.3);
\draw (9.2,5.3) node[above] {{\small $\delta$}};

\draw[red, <-> ] (10,5.3) -- (11.5,5.3);
\draw (10.7,5.3) node[above] {{\small $\delta$}};

\draw[blue, <-> ] (1.5,6.5) -- (8.5,6.5);
\draw (5,6.5) node[above] {{\small $r$}};

\draw[blue, <-> ] (11.5,6.5) -- (18.5,6.5);
\draw (15,6.5) node[above] {{\small $r$}};

\draw[red, <-> ] (0,5.3) -- (1.5,5.3);
\draw (0.7,5.3) node[above] {{\small $\delta$}};

\draw[red, <-> ] (1.5,5.3) -- (3,5.3);
\draw (2.2,5.3) node[above] {{\small $\delta$}};

\draw[red, <-> ] (17,5.3) -- (18.5,5.3);
\draw (17.7,5.3) node[above] {{\small $\delta$}};

\draw[red, <-> ] (18.5,5.3) -- (20, 5.3);
\draw (19.2,5.3) node[above] {{\small $\delta$}};

\end{tikzpicture}
\caption{Schematic picture of $\Delta, \Delta^-, \Delta_0, \Delta^+$.  We will usually take $r=\frac\tau{2n}$ and $\delta= \frac1{2m}$ for $n\ll m$, for example, later in the paper we use $m= n^{1+\beta'}$ below.}
\label{fig:Deltas}
\end{figure}
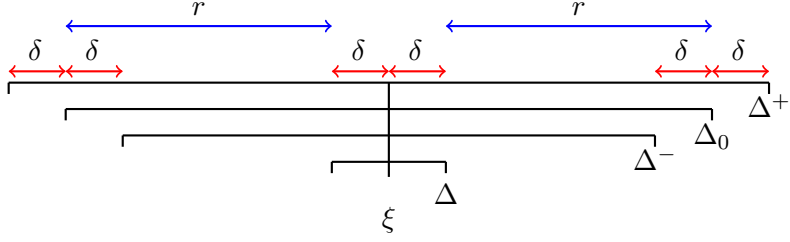

In the later steps of the proof of Theorem \ref{thm:poisson_like}, we take $r=\frac{\tau}{2n}$, and $\frac1m\ll r$, but at the moment we do not specify the dependence of $r$ and $m$ on $n$.
Writing $N_{t, n} (A,x):=\#\{j\in[t,n]:\,f^j(x)\in A\}$, the following result compares $\{R_n(r,x)=k\}$ with $\{N_{1, n}(\Delta_0,x)=k\}$.  
\begin{proposition}\label{prop.R-N}
The following inequality holds:
\begin{multline*}
\left|\mu(R_n(r,x)=k)-
\sum_{\Delta\in\mathcal{U}_m}\mu(\Delta\cap[N_{1, n}(\Delta_0,x)=k])\right|\\
\leq \mu(\tilde E_{g(n)}(r+2\delta))+C_1n\delta+C_2\Theta(g(n)).
\end{multline*}
where $C_1$ and $C_2$ are positive constants (independent of $n$, $m$ and $\delta$).
\end{proposition}
\begin{proof}
Given $\mathcal{U}_m$, we have the following decomposition of $\{ R_n(r,x)=k\}$
into a (disjoint) union of sets:
$$\{x\in\mathcal{X}: R_n(r,x)=k\}=\bigcup_{\Delta\in\mathcal{U}_m}\{\Delta\cap[R_n(r,x)=k]\}.$$ 
For each element of this union we have the following lemma. 
\begin{lemma}
In the notation above:
$$\{\Delta\cap[R_n(r,x)=k]\}\ominus \{\Delta\cap[N_{1, n}(\Delta_0,x)=k]\}
\subset \Delta\cap\left(\bigcup_{j=1}^{n}f^{-j}(\Delta^+\setminus\Delta^{-})\right).$$
Here $\ominus$ denotes symmetric difference. 
\end{lemma}
\begin{proof}
The proof of the lemma uses the definitions of $\Delta^{\pm}$ and $\Delta_0$. In particular, given
$j\geq 1$, suppose that $x\in\Delta$ and $f^j(x)\not\in\Delta^{+}$. 
Then clearly $f^j(x)\not\in\Delta_0$. Furthermore we have 
\begin{equation*}
\dist(f^jx,x)\geq |\dist(f^jx,\zeta)-\dist(\zeta,x)|> r.
\end{equation*}
Hence, the orbit at time $j$ gives zero count to both $R_n(r,x)$ and $N_{1,n}(\Delta_0,x)$. Similarly,
when $f^j(x)\in\Delta^{-}$, since $r>\delta$ we obtain a unit count contribution to both $R_n(r,x)$ and 
$N_{1, n}(\Delta_0,x)$.
In particular if the whole of the orbit of $x$ avoids $\Delta^+\setminus\Delta^{-}$, then
the events $\{R_n(r,x)=k\}$ and $\{N_{1, n}(\Delta_0,x)=k\}$ have the same itinerary of times
$1\leq t_1<t_2< \cdots <  t_k\leq n$ of unit count values (while all other times give zero count 
contributions). This completes the proof of the lemma.
\end{proof}
To prove Proposition \ref{prop.R-N}, we first of all have 
\begin{multline}\label{eq:R-N.est1}
\left|\sum_{\Delta\in\mathcal{U}_m}\mu(\Delta\cap(R_n(r,x)=k))-
\sum_{\Delta\in\mathcal{U}_m}\mu(\Delta\cap(N_{1, n}(\Delta_0,x)=k))\right|\\
\leq \sum_{\Delta\in\mathcal{U}_m}\mu\left(\Delta\cap
\left(\bigcup_{j=1}^{n}f^{-j}(\Delta^+\setminus\Delta^{-})\right)\right).
\end{multline}
On the right-hand side we make the following decomposition:
\begin{equation*}\label{eq:Delta-decomp}
\begin{split}
\bigcup_{j=1}^{n}f^{-j}(\Delta^+\setminus\Delta^{-}) &=
\left(\bigcup_{j=1}^{g(n)}f^{-j}(\Delta^+\setminus\Delta^{-})\right)\cup
\left(\bigcup_{j=g(n)+1}^{n}f^{-j}(\Delta^+\setminus\Delta^{-})\right)\\
&:=\Delta'\cup\Delta''.
\end{split}
\end{equation*}
Hence, the right-hand side of \eqref{eq:R-N.est1} becomes bounded by
\begin{equation*}
\sum_{\Delta\in\mathcal{U}_m}\left(\mu(\Delta\cap\Delta')+\mu(\Delta\cap\Delta'')\right).
\end{equation*}
By definition of $\tilde E_{g(n)}$ we see that
$$\sum_{\Delta\in\mathcal{U}_m}\mu(\Delta\cap\Delta')\leq\mu(\tilde E_{g(n)}(r+2\delta) ),$$
while by decay of correlations (of $\mathrm{BV}$ against $L^1$) we have
\begin{equation*}
\begin{split}
\sum_{\Delta\in\mathcal{U}_m}\mu(\Delta\cap\Delta'') &\leq\sum_{\Delta\in\mathcal{U}_m}\sum_{j=g(n)+1}^{n}
\left(\mu(\Delta)\mu(\Delta^+\setminus\Delta^{-})+\Theta(j)\|1_{\Delta}\|_{\mathrm{BV}} 
\|1_{\Delta^+\setminus\Delta^{-}}\|_{1}\right),\\
&\leq\sum_{\Delta\in\mathcal{U}_m} n\mu(\Delta)\mu(\Delta^+\setminus\Delta^{-})
+C\Theta(g(n))\delta\\
&\leq C_1n\delta+C_2\Theta(g(n)), 
\end{split}
\end{equation*}
where in the last step we have summed over $\Delta\in\mathcal{U}_m$, noting that the cardinality 
of $\mathcal{U}_{m}$ is $\delta^{-1}$.
The constants $C_1$ an $C_2$ are positive (independent of $n$, $m$ and $\delta$). This completes the proof.
\end{proof}
In the next step, we consider $\mu(\Delta\cap\{N_{1, n}(\Delta_0,x)=k\})$ and 
compare to $\mu(\Delta)\mu(N_{1, n}(\Delta_0,x)=k)$. We do this again
using the set $\tilde E_{g(n)}$ and decay of correlations. We have the following lemma
\begin{lemma}\label{lem:N-Delta.decorr}
Given $t>1$, the following decomposition holds (for constants $C_1,C_2>0$)
\begin{multline*}
\left|\sum_{\Delta\in\mathcal{U}_m}
\mu(\Delta\cap(N_{1,n}(\Delta_0,x)=k))-\sum_{\Delta\in\mathcal{U}_m}
\mu(\Delta)\mu(N_{1,n}(\Delta_0,x)=k)\right|\\
\leq 
C_1\delta^{-1}\Theta(t)+C_2\mu(\Delta_0)t+\mu(\tilde E_{t}(r+2\delta)).
\end{multline*}
\end{lemma}
\begin{proof}
First observe for $t,k\geq 1$ we have

\begin{align*}
\left\{N_{t, n}(\Delta_0)=k\right\}\sm\left\{ \ge 1 \text{ hit to $\Delta_0$ in time } [1,t)\right\} & \subset \{N_{1, n}(\Delta_0)=k\} \\
&\hspace{-5cm} \subset  \left\{N_{t, n}(\Delta_0)=k\right\} \cup\left\{ \ge 1 \text{ hit to $\Delta_0$ in time } [1,t)\right\} 
\end{align*}

and for $k=0$ we have
\begin{align*}
\{N_{t, n}(\Delta_0)=0\}\setminus &\{N_{1, n}(\Delta_0)=0\}\\
&=\{0\textrm{ hit to } \Delta_0 \text{ in time $[t,n]$, at least one in $[0,t)$}.\}.\\
\end{align*}
Hence,
\begin{multline*}
\left|\sum_{\Delta\in\mathcal{U}_m}\mu(\Delta\cap[N_{1, n}(\Delta_0,x)=k])-
\sum_{\Delta\in\mathcal{U}_m}\mu(\Delta\cap(N_{t,n}(\Delta_0,x)=k))\right|\\
\leq \mu(\tilde E_{t}(r+\delta)). 
\end{multline*}
Using decay of correlations, we have
\begin{multline*}
\left|\mu(\Delta\cap(N_{t, n}(\Delta_0,x)=k))-
\mu(\Delta)\mu(N_{1, n-t}(\Delta_0,x)=k)\right|\\
\leq
C\Theta(t)\|\mathbf{1}_{\Delta}\|_{\mathrm{BV}}
\|\mathbf{1}_{N_{1, n-t}(\Delta_0,k)}\|_{1},
\end{multline*}
for some $C>0$.
We can also replace $\{N_{1, n-t}(\Delta_0,x)=k\}$ by 
$\{N_{1,n}(\Delta_0,x)=k\}$. This introduces a further error term which can be quantified
using
\begin{align*}
\{N_{1, n}(\Delta_0,x)= k\} \sm \{ \ge 1\text{ hit to $\Delta_0$ in } (n-t,n]\} &\subset \{N_{1, n-t}(\Delta_0,x)= k\}\\
&\hspace{-6.5cm} \subset \{N_{1, n}(\Delta_0,x)= k\}\cup  \{ \ge 1\text{ hit to $\Delta_0$ in } (n-t,n]\}. 
\end{align*}
So the difference in these sets can be estimated by 
$\mu\left(\bigcup_{j=n-t}^{n}f^{-j}\Delta_0\right)\le t\mu(\Delta_0).$
We now sum over $\Delta\in\mathcal{U}_m$ to 
complete the proof of the lemma.
\end{proof}

Analysis of the set $N_{1,n}(\Delta_0,k)$ is now equivalent to a hitting time  
problem associated to the set $\Delta_0\equiv\Delta_0(\zeta_k)$, with $(\zeta_k)$
a uniformly spaced sequence in $\XX$, and with $k\in[1,m]$. Each set $\Delta_0$ corresponds to 
an interval of (Lebesgue) measure $=2r+2\delta$. Moreover the centre points $\zeta_k$ are independent
of the orbit of $x$. In the sequel we take $r=\tau/2n$ and choose $m=(2\delta)^{-1}$ accordingly.
As we can see from Lemma~\ref{lem:N-Delta.decorr}, some balancing is required on the speed at
which we allow $m\to\infty$ as $n\to\infty$ ($r\to 0$). We make this precise in the next section.

\subsection{Poisson approximation and error terms.}\label{sec.chen-stein}

In this section we consider the process $N_{n}(\Delta_0,x):=N_{0,n}(\Delta_0,x)$.
In conjunction with Proposition \ref{prop.R-N} and Lemma \ref{lem:N-Delta.decorr} we are
then able to determine the limit distribution for $R_n(r,x)$ when $r\equiv \frac{\tau}{2n}$.
Our analysis will also determine the error terms. The behaviour of these error terms will be useful 
in the proof of Theorem \ref{thm:almost-sure}. We will estimate the ``distance'' of $N_{n}(A,x)$ 
to a Poisson distribution
with parameter $n\mu(A)$. To do this, we define
the Total Variation distance $d_{TV}(U,V)$ between two positive integer valued random variables $U,V$ by
\begin{equation*}\label{eq:d-TV}
d_{TV}(U,V):=\frac{1}{2}\sum_{k=0}^{\infty}|\p(U=k)-\p(V=k)|.
\end{equation*} 
We have the following result, which is a re-statement of \cite[Theorem 2.1]{CC} in our context.
\begin{proposition}\label{prop:cs}
Consider the process $N_{n}(A,x)$, and $A\subset\XX$. Then 
$$d_{\mathrm{TV}}\left(N_n(A),\textrm{Poi}(n\mu(A))\right)\leq \mathcal{E}(A, n,p,M),$$
where for any $M\in \mathbb{N}$, 
$$\mathcal{E}(A, n,p,M):=2nM\{\mathcal{E}_1(A, n,p)+\mathcal{E}_2(A, p)\}+\mathcal{E}_3(A, n,p,M),$$
and
\begin{align*}
\mathcal{E}_1(A, n,p)& :=\sup_{j,q}\left|\mu(A\cap\{N_{p, n-j}(A)=q\})-
\mu(A)\mu(N_{p, n-j}(A)=q) \right|,\\
\mathcal{E}_2(A, p) &:=\sum_{j=1}^{p}\mu(A\cap f^{-j}(A)),\\
\mathcal{E}_3(A, n,p,M) &:=4\left(Mp\mu(A)(1+n\mu(A))+\frac{(n\mu(A))^M}{M!}e^{-n\mu(A)}+n\mu(A)^2\right),
\end{align*} 
where for $\mathcal{E}_1(A, n,p)$ the supremum is over $0\leq j\leq n-p$, $0\leq q\leq n-j-p$.
\end{proposition}
In the subsequent steps, we estimate the terms $\mathcal{E}_i$ (for $i=1,2,3$, suppressing notation where appropriate). The appropriate estimates will be based on our assumptions (A1), (A2). We remark that
corresponding terms can be also estimated under different dynamical assumptions on the regularity
of the measure and decay of correlations, see for example \cite{HW,PenSau16}.
Using Proposition \ref{prop.R-N} and 
Lemma \ref{lem:N-Delta.decorr} we then obtain the following result.
\begin{proposition}\label{prop.final-error}
\begin{equation}
\begin{split}\label{eq:final.error}
&\left|\mu\{R_n(r,x)=k\}-\sum_{\Delta\in\mathcal{U}_m} \mu(\Delta)\frac{(n\mu(\Delta_0))^k}{k!}
e^{-n\mu(\Delta_0)}\right|\\
&\quad\leq O(1)\left\{\mu(\tilde E_{g(n)}(r+\delta))+n\delta+\delta^{-1}\Theta(g(n))+g(n)\mu(\Delta_0)\right\}\\
&\quad\quad\quad+\sum_{\Delta\in\mathcal{U}_m}\mu(\Delta)\left(2nM\left[\mathcal{E}_1(\Delta_0)+\mathcal{E}_2(\Delta_0)\right]+\mathcal{E}_3(\Delta_0)\right),
\end{split}
\end{equation}
where $\mathcal{E}_i$ are corresponding error terms arising from Proposition \ref{prop:cs}
with $A=\Delta_0$ and appropriately chosen $p$ and $M$. 
\end{proposition}
We now choose suitable $r, t, m, M, p$. We choose $M=\log n$ and we put
$p=n^{\gamma}$ for some $\gamma\in(0,1)$. We take $t=(\log n)^2$ so that we can apply (A2),
and we take $m=n^{1+\beta'}$ for some $\beta'>0$ (so that $\delta=\frac{1}{2}n^{-1-\beta'}$).
For the value of $r$, we take $r=\tau/2n$ for $\tau>0$. From  \eqref{eq:final.error},
we see that there is some interdependence between how fast $n\to\infty$ and how fast we choose to 
take $\delta\to 0$. This is due to the requiring both $n\delta\to 0$, and
$\delta^{-1}\Theta(g(n))\to 0$ within the right-hand side of \eqref{eq:final.error}. For purposes
of proving Theorem~\ref{thm:poisson_like}, the choice of parameter rates above is appropriate.
 
Estimation of $\mathcal{E}_3$ uses information independent of the dynamics except for the regularity
of the measure $\mu$. Estimation of $\mathcal{E}_1$ depends on decay of correlations through (A1), 
while estimation of $\mathcal{E}_2$ utilises assumption (A2). We have the following results.
\begin{lemma}\label{lem.e3}
There exists $C>0$ and $c_1>0$, such that for $M=\log n$ and $p=n^{\gamma}$ (for $\gamma\in(0,1)$),
\begin{equation*}
\mathcal{E}_3(\Delta_0, n)\leq Cn^{-c_1}.
\end{equation*}
The constant $C$ depends on $\tau$.
\end{lemma}

\begin{proof}
Using our assumption that the invariant density $\rho$ is of bounded variation, we have 
$\mu(\Delta_0)\leq C|\Delta_0|\leq C\tau/n$ for some $C>0$. With our choice $M=\log n$,
and $p=n^{\gamma}$ we obtain
\begin{equation*}
\mathcal{E}_3(\Delta_0, n)\leq O(1)\left((\log n) n^{\gamma-1}\tau(1+\tau)+\frac{(C\tau)^{(\log n)}}{(\log n)!}
e^{-C\tau}+\frac{C^2\tau^2}{n}\right)\\
\end{equation*}
The intermediate term on the right can be estimated using Stirling's formula:
\begin{equation*}
\begin{split}
\frac{(C\tau)^{(\log n)}}{(\log n)!} &\approx (C\tau)^{\log n}(\log n)^{-\log n}e^{\log n}
(\log n)^{-\frac{1}{2}}\\
&\leq O(1)\exp\left\{\left(\log(C\tau)+1\right)\log n-(\log n)(\log\log n)\right\}, 
\end{split}
\end{equation*}
which tends to zero at a super-polynomial speed. Hence the dominant term for $\mathcal{E}_3$
is the quantity of order $n^{\gamma-1}\log n$. The conclusion of the Lemma now follows.
\end{proof}
\begin{remark}
For the proof of Theorem \ref{thm:almost-sure} (in Section~\ref{sec.pf.thm:almost-sure}) 
we will need to quantify this error term
in case $\tau$ depends on $n$, with $\tau\to\infty$ permitted. By inspecting the terms within the proof
Lemma \ref{lem.e3} we see that the same conclusion holds for the lemma provided 
$\tau\equiv\tau_n=o(\log n)$.
\end{remark}
We now consider $\mathcal{E}_1$.
\begin{lemma}\label{lem.e1}
There exists $C>0$ such that for $p=n^{\gamma}$ (and $\gamma\in(0,1)$),
\begin{equation*}
\mathcal{E}_1(\Delta_0, n)=o(n^{-a}).
\end{equation*}
for all $a>0$, (and hence superpolynomial decay).
\end{lemma}
\begin{proof}
For $\mathcal{E}_1$, we just use decay of correlations, and obtain
\begin{equation*}
\mathcal{E}_1(\Delta_0, n)\leq \Theta(p)\|1_{A}\|_{\mathrm{BV}}\|1_{(N_{p, n-j}(A)=q)}\|_{1}\leq 2\Theta(n^{\gamma}).
\end{equation*}  
This error term clearly goes to zero at a super-polynomial speed by assumption on exponential
decay for the rate function $\Theta$.
\end{proof}
\begin{remark}
The error terms $\mathcal{E}_1$ and $\mathcal{E}_3$ do not depend on $\Delta$, and hence can be taken
uniform over $\Delta\in\mathcal{U}_m$.
\end{remark}
Finally we consider $\mathcal{E}_2$. We have the following lemma.
\begin{lemma}\label{lem.e2}
There exists $C>0$, and $c_2>0$ such that for $p=n^{\gamma}$ (and $\gamma\in(0,1)$)
such that
\begin{equation*}
\sum_{\Delta\in\mathcal{U}_m}\mu(\Delta)\mathcal{E}_2(\Delta_0, p)=O(n^{-c_2-1}),
\end{equation*}
with implied constant within $O(\cdot)$ depending on $\tau$.
\end{lemma}
\begin{proof}
For $\mathcal{E}_2$ we can write
\begin{equation}\label{eq.e2-split}
\mathcal{E}_2(A, p)=\sum_{j=1}^{t}\mu(A\cap f^{-j}(A))+\sum_{j=t+1}^{p}\mu(A\cap f^{-j}(A)),
\end{equation}
where we put $t=(\log n)^2$, and note that $A=\Delta_0$ is a small interval of measure $O(\tau/2n)$,
with a centre point $\zeta_k$. Recall $\zeta_k$ are uniformly spaced points in $\XX$, with 
$1\leq k\leq m$. In \eqref{eq.e2-split}, we have for the second term on the right hand side:
\begin{equation*}
\begin{split}
\sum_{j=t+1}^{p}\mu(A\cap f^{-j}(A)) &\leq \sum_{j=t+1}^{p}\left(
\Theta(j)\|1_{A}\|_{\mathrm{BV}}\|1_{A}\|_{1}+\mu(A)^2\right)\\
&\leq n^{\gamma}\Theta((\log n)^2)+n^{\gamma}(\tau/2n)^2=O(n^{-1-c'}),
\end{split} 
\end{equation*}
for some $c'>0$, where we have used (A1). We note that the implied constant in $O(\cdot)$ 
depends on $\tau$. This term causes no complication, and its estimation does not depend on $\zeta_k$

Now the first term on the right hand side of \eqref{eq.e2-split} depends on $\zeta_k$. Indeed
if $\zeta_k$ is a periodic point then this error term will not be small. To bound the error term involving
$\mathcal{E}_2$ within equation 
\eqref{eq:final.error}, it is sufficient to estimate 
\begin{equation}\label{eq:short_times}
\sum_{\Delta\in \U_m}\mu(\Delta)\sum_{j=1}^{t}\mu(\Delta_0\cap f^{-j}\Delta_0).
\end{equation}
We'll bound the sum directly using condition (A2).
By assumption $|\Delta_0|\sim \tau/n$. Hence, letting $\zeta_k$ denote the centre of $\Delta$, for $x\in \Delta_0\cap f^{-j}\Delta_0$ we have for each $j$,
$$d(x,f^jx)\leq d(x,\zeta_k)+d(f^jx,\zeta_k)<\frac{2\tau}{n}.$$
Therefore for $a<1$ we have
$x\in \Delta_0\cap E_{j}(n^{-a})$,
for all $n$ sufficiently large so that $2\tau<n^{1-a}.$
For $b>0$, let
$$\mathcal{F}_n^+:=\{\Delta\in\mathcal{U}_m:\mu(\Delta_0\cap \tilde E_g(n^{-a}))\geq n^{-1-b}\},$$
with $g(n)=a^2(\log n)^2$.
Since the density $\rho(x)$ lies in $\mathrm{BV}$ we have $\mu(\Delta)\leq C|\Delta|=C/m$ for some
$C>0$ and independent of $\Delta$. Thus \eqref{eq:short_times} can be bounded by
\begin{equation}
\sum_{\Delta\not\in\mathcal{F}_n^{+}} \frac{C}m\sum_{j=1}^t\mu(\Delta_0\cap f^{-j}\Delta_0)
+\sum_{\Delta\in\mathcal{F}_n^{+}}\frac Cm \sum_{j=1}^t\mu(\Delta_0\cap f^{-j}\Delta_0).
\label{eq:F+sums}
\end{equation}
We can bound the first term by $\frac{C}{m}(\#(\mathcal{U}_m\setminus\mathcal{F}^{+}_n))n^{-1-b}t\le Ca^2(\log n)^2n^{-1-b}$.
For the second sum in \eqref{eq:F+sums}, observe that for $\beta_1$ from (A2) we have
\begin{equation}
\begin{split}\label{eq:cover-split}
Cn^{-a\beta_1} &\geq \mu(\tilde{E}_{g}(n^{-a}))=\mu\left(\bigcup_{\Delta}\left\{x\in\Delta:\,\exists\,j\leq g(n^a) \text{ where } d(x,f^jx)\leq\frac{1}{n^a}\right\}\right)\\
&\geq C'\frac{n}{2m\tau}
\sum_{\Delta\in\mathcal{F}^{+}_n}\mu(\Delta_0\cap \tilde E_{g}(n^{-a}))
\geq C'\frac{n}{2m\tau}n^{-1-b}\#\mathcal{F}^{+}_n,
\end{split}
\end{equation}
for some $C'>0$. Obtaining the second line in equation \eqref{eq:cover-split} can be explained as follows, see also \cite[Section 4]{Collet}. We cover
$\XX$ using all $\Delta\equiv\Delta(\zeta_k)$ in $\mathcal{F}^{+}_n$, with uniformly spaced
centres $\zeta_k$, $k\in[1,m]$.
For each $\Delta\in\mathcal{F}^{+}_n$ there is an expanded
set $\Delta_0$ centred on $\Delta$, with $|\Delta_0|=\tau/n$.
An \emph{under-cover} of $\XX$ (and hence $\tilde E_g(n^{-a})$)
can be obtained by taking \emph{disjoint} sets of the form $\Delta_0(\zeta_{k_j})$,
where the centres $\{\zeta_{k_j}\}$ are uniformly spaced, and $k_{j+1}-k_j=2\lfloor \tau m/ n \rfloor$.
(The additional multiplying factor of 2 is to ensure an under-covering using disjoint sets).
For such an (under)-covering we obtain
$$\mu(\tilde{E}_{g}(n^{-a}))\geq \sum_{k_j\leq m}\mu(\Delta_0(\zeta_{k_j})\cap \tilde{E}_{g}(n^{-a})).$$ 
If we range over all such $\Delta\in\mathcal{F}^{+}_n$ (i.e. eventually all such $\zeta_{k}$
for $k\in[1,m]$) then we obtain at most 
$2\lfloor \tau m/ n \rfloor$ such under-covers, and this leads to the estimate
$$2\left\lfloor\frac{\tau m}{n} \right\rfloor\mu(\tilde{E}_{g}(n^{-a}))
\geq \sum_{\Delta\in\mathcal{F}^{+}_n}\mu(\Delta_0(\zeta_{k})\cap 
\tilde{E}_{g}(n^{-a})),$$
and hence to the estimate stated in equation \eqref{eq:cover-split}.
We therefore obtain
$$\#\mathcal{F}^{+}_n\leq C\tau m n^{b-a\beta_1},$$
for some $C>0$.
Thus for the second term in \eqref{eq:F+sums} we obtain
\begin{equation*}
\begin{split}
\sum_{\Delta\in\mathcal{F}^{+}_n}\frac{1}{m}\sum_{j=1}^{g(n)}\mu(\Delta_0\cap f^{-j}\Delta_0)
&\lesssim \sum_{\Delta\in\mathcal{F}^{+}_n}\frac{1}{m}\sum_{j=1}^{g(n)}|\Delta_0| 
\leq  \#\mathcal{F}^{+}_n
 \frac{g(n)}m |\Delta_0| \\
&\lesssim\#\mathcal{F}^{+}_n g(n)\frac{|\Delta_0|}m
\leq   C\tau n^{b-a\beta_1} a^2(\log n)^2 |\Delta_0|\\
&\leq  2C \tau^2n^{b-1-a\beta_1} a^2(\log n)^2.
\end{split}
\end{equation*}
To complete the proof of Lemma \ref{lem.e2}, first choose $a<1$, and notice that $\Theta(g(n))$ still 
tends to zero at a superpolynomial speed when adjusting $g(n)$ from $(\log n)^2$
to $(a\log n)^2$. There is flexibility on choosing $a<1$, and we can take this parameter 
arbitrarily close to 1. We then choose $b=\frac{1}{2}a\beta_1$. 
\end{proof}

\subsection{Completing the proof of Theorem \ref{thm:poisson_like}.}
Using Lemmas~\ref{lem.e3}-\ref{lem.e2}, we obtain
\begin{equation}\label{eq:R-sum.est}
\left|\mu\{R_n(n,r_n)=k\}-\sum_{\Delta\in\mathcal{U}_m} \mu(\Delta)\frac{(n\mu(\Delta_0))^k}{k!}
e^{-n\mu(\Delta_0)}\right|
\leq Cn^{-\alpha'},
\end{equation}
where we take $r=\tau/2n$ and $m=n^{1+\beta'}$. In \eqref{eq:R-sum.est} we have 
$C>0$ and $\alpha'>0$. We now simplify the sum over $\mathcal{U}_m$. Lebesgue differentiation gives
\begin{align*}
\mu(\Delta) &=|\Delta|\rho(\zeta_k)+o_{\Delta}(1),\\
\mu(\Delta_0) &=|\Delta_0|\rho(\zeta_k)+o_{\Delta_0}(1),
\end{align*}
where for a generic (full measure) centre point $\zeta_k$ the quantity $o_{A}(1)$ satisfies 
$o_{A}(1)/|A|\to 0$ as $|A|\to 0$ (taking $A$ either as $\Delta$ or $\Delta_0$). For density $\rho\in\mathrm{BV}$,
and for non-generic centres (such as $\rho$ having a jump discontinuity at $\zeta_k$), the corresponding limits are uniformly bounded from above. 
Hence to treat the limit of the sum in equation \eqref{eq:R-sum.est}, define the sequence of 
functions $g_m(x)$ given by 
$g_m(x)=\frac{(\tau\rho(\zeta_k))^k}{k!}e^{-\tau\rho(\zeta_k)}$,
where $x\in\Delta(\zeta_k)$ and $\Delta(\zeta_k)\in\mathcal{U}_m$. Then for $\mu$ almost every $x\in\XX$, 
the functions $g_m(x)\to g(x)$ pointwise
as $m\to\infty$, with $g(x)=\frac{(\tau\rho(x))^k}{k!}e^{-\tau\rho(x)}$, and $g\in L^1(\mathrm{Leb})$.
The Dominated Convergence Theorem gives
\begin{equation*}
\lim_{n\to\infty}\sum_{\Delta\in\mathcal{U}_m} \mu(\Delta)\frac{(n\mu(\Delta_0))^k}{k!}e^{-n\mu(\Delta_0)}
=\int_{\mathcal{X}}\frac{(\tau\rho(x))^k}{k!}e^{-\tau\rho(x)}\rho(x)\,dx.
\end{equation*}
This completes the proof of Theorem \ref{thm:poisson_like}.

\section{proof of Theorem \ref{thm:almost-sure}}\label{sec.pf.thm:almost-sure}
The proof of Theorem \ref{thm:almost-sure} is as follows. We first treat case (b).
In the case that $r_n\to 0$ is a monotone sequence, an elementary observation is
that 
$$\mu\left\{x\in\XX:\, \min_{j\leq n}d(f^jx,x)<r_n\,\mathrm{i.o.}\right\}=
\mu\left\{x\in\XX:\, d(f^nx,x)<r_n\,\mathrm{i.o.}\right\}.$$
For the equation above, clearly the set measured on the right is contained in the set measured on the left. For the other direction see \cite{Galambos,HKKP}.
From \cite[Theorem C]{KKP2}, it can be shown that
$$\sum_{n}\int_{\XX}\mu(B(x,r_n))\,\mathrm{d}\mu<\infty\;\implies\;
\mu\{x\in\XX:\, d(f^nx,x)<r_n\,\mathrm{i.o.}\}=0.$$
Hence, by considering the complement of 
$\{x\in\XX:\, d(f^jx,x)<r_n,\,\mathrm{i.o.}\}$ the result for case (b) immediately follows.

We now turn to case (a). We follow step by step the proof
of Theorem \ref{thm:poisson_like} as applied to the particular case $\{R_n(r,x)=0\}$. Recalling
that $r=\frac\tau{2n}$ here we obtain
$$\lim_{n\to\infty}\mu\left(x\in\XX:\,\min_{k\leq n}\dist(f^kx,x)>\frac{\tau}{2n}\right)
=\int_{\XX}\rho(x)e^{-\tau\rho(x)}\,dx.$$
However, the result above is not enough on its own, as we need to estimate the error rate for the convergence to the limit. To estimate the error rate we consider the cases (a)(i) and (a)(ii)
stated in the Theorem~\ref{thm:almost-sure} separately. However,
the assumptions relating to either of these cases will only appear in the final steps of the proof.
For the sequence $r=\frac\tau{2n}$,
we take a slower decaying sequence for $r\equiv r_n$ of
the form $r=\tau(n)/(2n)$ with $\tau(n)\to\infty$. In case (a)(i) this is equivalent to taking 
$\tau=2c\log\log n$, while in case (a)(ii) we take $\tau=(\log n)^{\gamma}$. The relevant
error rates in the convergence can be obtained from  \eqref{eq:final.error}. We estimate
each of the error terms in turn with appropriate choice of parameter sequence. In particular
we take $t=(\log n)^2$ for the correlation time-lag, $m=n^{1+\beta'}$ for the cardinality of
$\mathcal{U}_m$. To obtain conclusions similar to those given in Lemmas ~\ref{lem.e3}-\ref{lem.e2} we again take $p=n^{\gamma}$ and $M=\log n$. Using again (A1) and (A2), we see that all the terms within the right-hand side of \eqref{eq:final.error} all decay at a power law rate in $n$, thus leading to the same conclusion as given in \eqref{eq:R-sum.est}. This leads to
\begin{equation}
\left|\mu(R_n(r,x)=0)-\sum_{\Delta\in\mathcal{U}_m} \mu(\Delta)e^{-n\mu(\Delta_0)}\right|
\leq Cn^{-\alpha'}
\label{eq:asmain}
\end{equation}

The idea now is to use a First Borel-Cantelli argument to show that for a certain subsequence
$r\equiv r_{n_k}$, we have $\mu(R_{n_k}(r_{n_k},x)=0\,\mathrm{i.o.})=0.$ This requires an estimate in terms of $n$ for the sum over $\mathcal{U}_m$. We
therefore require assumptions on the regularity of the density $\rho(x)$ for $x\in\Delta_0$.
By Lebesgue differentiation, we recall
\begin{align*}
\mu(\Delta) &=|\Delta|\rho(\zeta_k)+o_{\Delta}(1),\\
\mu(\Delta_0) &=|\Delta_0|\rho(\zeta_k)+o_{\Delta_0}(1).
\end{align*}
Under the assumptions of the theorem, recall that $\rho$ is uniformly H\"older (and of
bounded variation). In the following sequence of estimates, we keep the steps as general as possible,
using the H\"older assumptions in the final steps.
If $\rho(x)$ is of bounded variation, then
$$o_{\Delta}(1)\lesssim |\Delta|\mathrm{var}_{\Delta}(\rho),
\quad o_{\Delta_0}(1)\lesssim |\Delta_0|\mathrm{var}_{\Delta_0}(\rho).$$
Potential problems are that the corresponding variations are not asymptotically small
with $|\Delta|$ and (resp.) $|\Delta_0|$. For $r>0$ let
$$G_r(x):=\int_{x-r}^{x+r}\rho(z)\,dz.$$
Clearly, when $r=\tau/2n$, $G(\zeta_k)=\mu(\Delta_0)$. The following lemma
will be useful in the case where we take $\tau\to\infty$ as $n\to\infty$, such as in the
case $\tau_n=c\log\log n$.
\begin{lemma}\label{lem:bv-est}
Suppose that for $\beta'>0$ we have $n/m\leq n^{-\beta'}$. Then there exists 
$\alpha_2>0$ such that for all $r>0$ we have
\begin{equation*}
\left|\sum_{\Delta\in\mathcal{U}_m} \mu(\Delta)e^{-n\mu(\Delta_0)}-
\int_{\XX}\rho(x)e^{-nG_r(x)}\,dx\right|\leq Cn^{-\alpha_2}.
\end{equation*}
\end{lemma}
\begin{proof}
To prove the lemma, first of all we have
\begin{multline}\label{eq:BV-asymp1}
\left|\sum_{\Delta\in\mathcal{U}_m} \mu(\Delta)e^{-n\mu(\Delta_0)}-\sum_{k=1}^{m}
|\Delta(\zeta_k)|\rho(\zeta_k)e^{-nG_r(\zeta_k)}\right|\\
\lesssim 
\sum_{k=1}^{m}|\Delta_k|\mathrm{var}_{\Delta}(\rho)
e^{-nG_r(\zeta_k)}. 
\end{multline}
This follows using (for general $\Delta$), and any $x\in\Delta$,
$$\bigl|\mu(\Delta)-\rho(x)|\Delta|\bigr|\leq|\rho(x^+)-\rho(x^{-})||\Delta|
\leq|\Delta|\mathrm{var}_{\Delta}(\rho).$$  
Here, $\rho(x^+),\rho(x^{-})$ correspond to max/min values of $\rho$ on $\Delta$. 
In \eqref{eq:BV-asymp1} above we recall that $\Delta_k\equiv\Delta(\zeta_k)$.
Each $\Delta_k$ has uniform size $|\Delta|=1/m$, and from the fact that $\rho\in\mathrm{BV}$, then there exists $C>0$ such that $\mathrm{var}(\rho)\leq C$. Using the fact that
$e^{-nG_r(\zeta_k)}\leq 1$, it follows that
\begin{equation*}
\sum_{k=1}^{m}|\Delta_k|\mathrm{var}_{\Delta}(\rho)
e^{-nG_r(\zeta_k)}\leq |\Delta|\sum_{\Delta\in\mathcal{U}_m}\mathrm{var}_{\Delta}(\rho)<\frac{C}{m}.
\end{equation*}
For $x\in\Delta$, we compare $\rho(\zeta_k)e^{-nG_r(\zeta_k)}$ to $\rho(x)e^{-nG_r(x)}$. This
comparison requires only assumptions on the variation of $\rho$ over $\Delta$. 
For any $\Delta\in\mathcal{U}_m$, and all $x\in\Delta$ we observe that
\begin{multline}\label{eq:int_est_split1}
\left|\int_{\Delta}\rho(z)e^{-nG_r(z)}\,dz-\rho(x)e^{-nG_r(x)}|\Delta|\right|\\
\leq |\Delta|\sup_{\Delta}(\rho)\left(e^{-nG_r(x^{+})}-
e^{-nG_r(x^{-})}\right)+
|\Delta|\sup_{\Delta}\left(e^{-nG_r(x)}\right)\mathrm{var}_{\Delta}(\rho),
\end{multline}
where $x^{\pm}$ correspond to values where $G_r(x)$ takes its max/min values in $\Delta$.
We estimate the first term on the right-hand side of \eqref{eq:int_est_split1}.
 There exist $C_1,C_2>0$ such that
$$\sup_{\Delta}(\rho)<C_1,\quad\sup_{\Delta}\left(e^{-nG_r(z)}\right)<C_2.$$
In fact we can take $C_2=1$. Since $z\mapsto e^{-nz}$ is differentiable, the  Mean Value Theorem implies
$$\left|e^{-nG_r(x^{+})}-
e^{-nG_r(x^{-})}\right|
\leq\sup_{\Delta}\left(ne^{-nG_r(x)}\right)|G_r(x^+)-G_r(x^-)|.$$
This leads to elementary bounds (assuming without loss of generality that $x^{+}>x^{-}$):
\begin{equation*}
\begin{split}
|G_r(x^+)-G_r(x^-)| &\leq \left|\int_{x^{+}-r}^{x^{+}+r}\rho(z)\,dz-
\int_{x^{-}-r}^{x^{-}+r}\rho(z)\,dz\right|\\
&\leq \int_{x^{-}+r}^{x^{+}+r}\rho(z)\,dz+\int_{x^{-}-r}^{x^{+}-r}\rho(z)\,dz\\
&\leq 2\sup_{\Delta_0}(\rho)\cdot |x^{+}-x^{-}|\leq C|\Delta|,
\end{split}
\end{equation*}
since $\rho$ is uniformly bounded on $\Delta_0$ (independently of $r$). 
We obtain for all $x\in\Delta$
\begin{equation}\label{eq:int_est_split2}
\left|\int_{\Delta}\rho(z)e^{-nG_r(z)}\,dz-\rho(x)e^{-nG_r(x)}|\Delta|\right|
\leq C'n|\Delta|^2+C''|\Delta|\mathrm{var}_{\Delta}(\rho).
\end{equation}
Combining  \eqref{eq:BV-asymp1} and \eqref{eq:int_est_split2}, and summing
over $\Delta\in\mathcal{U}_m$ we obtain (for updated uniform constants $C', C''$)
\begin{multline*}
\left|\sum_{\Delta\in\mathcal{U}_m} \mu(\Delta)e^{-n\mu(\Delta_0)}-
\int_{\XX}\rho(x)e^{-nG_r(x)}\,dx\right|\\
\leq \sum_{\Delta\in\mathcal{U}_m}\left(C'n|\Delta|^2+C''|\Delta|\mathrm{var}_{\Delta}(\rho)\right).
\end{multline*}
Since $\rho\in\mathrm{BV}$ we clearly have $\sum_{\Delta}\mathrm{var}_{\Delta}(\rho)<\mathrm{var}(\rho)<C$.
Hence for $n/m\leq n^{-\beta'}$, we obtain the proof of the Lemma.
\end{proof}
\begin{remark}
In the proof of Lemma \ref{lem:bv-est}, the dependence on $r$ does not explicitly feature. The error
terms just depend on the spacing sizes $|\Delta|$ associated to the distribution of the 
centre points $\{\zeta_k\}$. In the case where $\inf_{x\in\Delta}\rho(x)=0$
the estimation of $G_r(x)$ is more delicate, and requires assumptions
on how the density function $\rho$ varies on $\Delta_0$, so that we can control 
$\mathrm{var}_{\Delta_0}G_r$. Note that for the proof of Lemma \ref{lem:bv-est}, we did not need to study
the variation of $G_r$ over $\Delta_0$. It was sufficient to control the difference between
the maximum and minimum of $G$ over small intervals $\Delta$.
\end{remark}

We are now in a position to consider cases (a)(i) and (a)(ii). First consider case (i), and assume that 
the invariant density is bounded away from zero. 
Using the fact that
$G_r(x)\geq 2r\inf_{x\in\XX}\rho(x),$ we obtain
\begin{equation*}
\int_{\XX}\rho(x)e^{-nG_r(x)}\,dx\leq C_1\exp\{-2nr\inf_{x\in\XX}\rho(x)\}.
\end{equation*}
Using Lemma~\ref{lem:bv-est}, \eqref{eq:asmain}, and going along the sequence $r_n=c\log\log n/n$ we obtain
\begin{equation}\label{eq.bc-error1}
\mu(R_n(r_n,x)=0)\leq Cn^{-\alpha_2}+C_1\exp\{-2c(\log\log n)\inf_{x\in\XX}\rho(x)\},
\end{equation}
for some $\alpha_2>0$. 
As it stands, the right-hand side of \eqref{eq.bc-error1} is not summable. Hence
we now go along the subsequence $n_k=\lfloor a^k \rfloor$ for $a>1$. This has the property that 
$n_{k+1}/n_k\to a$ as $n\to\infty$.
Putting $\rho_\mu=\inf_{x\in\XX}\rho(x)$, we obtain
\begin{equation*}\label{eq:as2}
\begin{split}
\mu(R_{n_k}(r_{n_k},x)=0) &\leq C_1(\log n_k)^{-2c\rho_\mu}
+O(n_{k}^{-\alpha_2}),\\
&\leq C_1(k\log a)^{-2c\rho_\mu}+O(a^{-k\alpha_2})
\end{split}
\end{equation*}
Since $a>1$, the right-hand side is summable provided $2c\rho_\mu>1$.
This matches the assumption within Theorem \ref{thm:almost-sure} for the value of $c$.
Hence, there exists $k_0$ such that for all $k\geq k_0$ we have
$$\dist(f^{n_k}x,x)<c\log\log n_k/n_k.$$
Consider $n\in[n_{k},n_{k+1}]$ for $k>k_0$. Since $n_{k+1}/n_k\to a$ we have 
$n_{k}>(1+o(1))n/a$. Therefore, for all sufficiently large $n$, 
$$\dist(f^{n}x,x)<\frac{c\log\log (n/a)}{(n/a)}(1+o(1)).$$
Since $a$ can be arbitrarily chosen close to 1 the result follows for case (a)(i).

Now consider case (a)(ii), where by assumption $\rho(x)$ is uniformly H\"older. 
Hence, there exist $\eta\in(0,1]$, $L>0$ and $\delta>0$ such that $\forall\,x,y$ with $|x-y|<\delta$
we have
$$|\rho(x)-\rho(y)|\leq L|x-y|^{\eta}.$$
Thus, we have refined asymptotic information on $\mu(\Delta)$ and $\mu(\Delta_0)$ as follows:
\begin{align*}
\left|\mu(\Delta)-|\Delta|\rho(\zeta_k)\right| &\leq L|\Delta|^{1+\eta},\\
\left|\mu(\Delta_0)-|\Delta_0|\rho(\zeta_k)\right| &\leq L|\Delta_0|^{1+\eta}.
\end{align*}
Thus we now take $m=n^{1+\beta'}$, with $\beta'\in(0,1)$ and
$\tau=(\log n)^{\gamma}$ for some $\gamma>0$. This choice of $\tau$ corresponds to $r=(\log n)^{\gamma}/2n$.
Hence, for all $x\in\XX$ we have
$$\left|G_{r}(x)-\frac{(\log n)^{\gamma}}{n}\rho(x)\right|\leq L'\frac{(\log n)^{\gamma(1+\eta)}}{
n^{1+\eta}},$$
and moreover
$$\left|e^{-nG_r(x)}-e^{-(\log n)^{\gamma}\rho(x)}\right|\leq\frac{1}{n^{\gamma'}},$$
for some $\gamma'>0$.
Therefore, using Lemma~\ref{lem:bv-est} and \eqref{eq:asmain}, we obtain
\begin{equation}\label{eq.bc-error1b}
\mu(R_n(r_n,x)=0)\leq Cn^{-\alpha_2}+\int_{\mathcal{X}}\rho(x)e^{-(\log n)^{\gamma}\rho(x)}\,dx.
\end{equation}
In general, the right-hand side of \eqref{eq.bc-error1} is not summable (for example when $\gamma<1$).
To extend to these cases,
we can again go along the subsequence $n_k=\lfloor a^k \rfloor$ for $a>1$.
This leads to
\begin{equation*}\label{eq:as2b}
\mu(R_{n_k}(r_{n_k},x)=0) 
\leq\int_{\mathcal{X}}\rho(x)e^{-(k\log a)^{\gamma}\rho(x)}  \,dx +O(a^{-k\alpha_2}).
\end{equation*}
Since $a>1$, the right-hand side is summable if
$$\sum_{k\geq 1}\int_{\mathcal{X}}\rho(x)e^{-(k\log a)^{\gamma}\rho(x)}\,dx <\infty.$$
This matches the assumption of case (a)(ii) by taking $\epsilon=(\log a)^{\gamma}$. 
The proof is then completed using the same argument as in case (a)(i).

\section{Proof of Theorem \ref{thm.asym-hyp-rec} }\label{pf:thm.asym-hyp-rec}
For the distributional limit law, it is sufficient to get the appropriate estimates for $\mu\left(x\in \XX: R_n(r_n, x)\le k\right)$ consistent with Theorem~\ref{thm:poisson_like}. For almost sure results,
we obtain almost-sure upper/lower bounds for $\min_{j\leq n}X_j(x)$ consistent with the cases described
in Theorem~\ref{thm:almost-sure}.

For $\eps>0$, consider $(Y_\eps, F_\eps= f^{\tau_\eps}, \mu_\eps)$  as in the definition of asymptotically hyperbolic.
Let $R_n^{F_\eps}$ be the counting function for the random variable $Y_i = \text{dist}(F_\eps^ix, x)$.  Then by assumption Theorem~\ref{thm:poisson_like} applies to this counting function on $(Y_\eps, F_\eps, \mu_\eps)$.

Given $\delta>0$ and $N\in \mathbb{N}$, define 
$$Y_\eps^{\delta, N}:= \left\{x\in Y_\eps: \left|\frac{\tau_\eps^k(x)}k-\frac1{\mu(Y_\eps)}\right|<\delta \text{ for all } k \ge N\right\}.$$
Here, $\tau_\eps^k(x)=\sum_{j=0}^{k-1}\tau_{\eps}(F_\eps^jx)$.
By the ergodic theorem, $\mu_\eps(Y_\eps^{\delta, N}) \to 1$ as $N\to \infty$.

Define $g^\pm(\eps, \delta):=\left(\frac1{\mu_\eps(Y)}\pm\delta\right)$.  So given $x\in Y_\eps^{\delta, N}$, we can write 
$$ng^-(\eps, \delta)<\tau_\eps^n(x)<ng^+(\eps, \delta).$$
So if the $f$-iteration time of $x\in Y_\eps^{\delta, N}$ is $n$, then the $F_\eps$ iteration time is in $\left( \frac n{g^+(\eps, \delta)}, \frac n{g^-(\eps, \delta)}\right)$.  We also note that for every $k$ hits of $x\in Y_\eps$ to an $r_n$-neighbourhood of $x$ by $f$ up to time $\tau_\eps^n(x)$, there can only be fewer hits by $F_\eps$ up to time $n/g^-(\eps, \delta)$.  

So we deduce that 
\begin{align*}\mu\left(x\in Y_\eps^{\delta, N}: R_n(r_n, x)\le k\right) & \le \mu\left(x\in Y_\eps^{\delta, N}: R_{ n/{g^-(\eps, \delta)}}^{F_\eps}(r_n, x)\le k\right)\\
 & \le \mu\left(x\in Y_\eps: R_{n/g^-(\eps, \delta)}^{F_\eps}(r_n, x)\le k\right).
 \end{align*}

Moreover, for every $k$ hits by $F_\eps$ of $x\in Y_\eps$ to an $r_n$-neighbourhood of $x$ up to time $n/g^+(\eps, \delta)$ there can be at most $k$ corresponding hits by $f$ of $x$ to its $r_n$-neighbourhood, if we exclude hits at the edge of $Y_\eps$.  Taking $Y_{\eps, r_n}$ to be $Y_\eps$ with a $r_n$-neighbourhood removed,
\begin{align*}
\mu\left(x\in Y_\eps^{\delta, N}: R_{n/g^+(\eps, \delta)}^{F_\eps}(r_n, x)\le k\right) &\\
&\hspace{-4cm} \le \mu(Y_\eps\setminus Y_{\eps, r_n})+ \mu\left(x\in Y_{\eps, r_n}\cap Y_\eps^{\delta, N}: R_{n}(r_n, x)\le k\right),
\end{align*}
so 
\begin{align*}&\mu\left(x\in Y_\eps: R_{n/g^+(\eps, \delta)}^{F_\eps}(r_n, x)\le k\right) - \mu(Y_\eps\sm Y_\eps^{N, \delta})\\
&\hspace{3cm} \le  \mu(Y_\eps\setminus Y_{\eps, r_n})+ \mu\left(x\in  Y_\eps^{\delta, N}: R_{n}(r_n, x)\le k\right).
\end{align*}

For the upper bound above we can write
\begin{align*}\mu\left(x\in Y_\eps^{\delta, N}: R_n(r_n, x)\le k\right) 
 & \le \mu(Y_\eps)\mu_\eps\left(x\in Y_\eps: R_{n/g^-(\eps, \delta)}^{F_\eps}(r_n, x)\le k\right).
 \end{align*}
Since by Theorem~\ref{thm:poisson_like}, 
\begin{align*}
\mu_\eps\left(x\in Y_\eps: R_{n/g^-(\eps, \delta)}^{F_\eps}(r_n, x)= j\right)& \\
&\hspace{-2cm}\to  \int_{Y_\eps}\frac{\left({\tau}{g^-(\eps, \delta)}\right)^j\rho_\eps(x)^{j+1}}{j!}
e^{-\rho_\eps(x){\tau}{g^-(\eps, \delta)}}\,dx.\end{align*}
for $\rho_\eps = \rho/\mu(Y_\eps)$, the asymptotic estimate becomes
\begin{align*}\mu\left(x\in \XX: R_n(r_n, x)\le k\right) & \\
&\hspace{-3.8cm}
\lesssim \mu((Y_\eps^{\delta, N})^c)+ \mu(Y_\eps)\sum_{j=0}^k  \int_{Y_\eps}\frac{\left({\tau}{g^-(\eps, \delta)}\right)^j\rho_\eps(x)^{j+1}}{j!}
e^{-\rho_\eps(x){\tau}{g^-(\eps, \delta)}}\,dx.
\end{align*}
We can take the first term to 0 by letting $N$ get large.  Then $g^\pm(\eps, \delta)\to 1$ $Y_\eps\to \XX$ and $\rho_\eps(x)\to \rho(x)$ as $\eps\to 0$, so 
\begin{align*}\mu\left(x\in \XX: R_n(r_n, x)\le k\right) &
\le \sum_{j=0}^k  \int_{\XX}\frac{\tau^j\rho(x)^{j+1}}{j!}
e^{-\rho(x)\tau}\,dx.
\end{align*}
The lower bound follows similarly, thereby completing the proof of Theorem \ref{thm.asym-hyp-rec}
relevant for the distributional limit law.

For the corresponding almost-sure results the same ideas apply. We consider each $Y_{\eps}$, and for the various cases (a)(i), (a)(ii), and (b) in Theorem~\ref{thm:almost-sure} we just replace $r_n$ by the corresponding $r_{n/g^{\pm}(\eps,\delta)}$, and similarly replace the density $\rho(x)$ by 
$\rho_{\epsilon}(x)$. The corresponding range of $c$ in case (a)(i) depends on having
$\inf_{x\in\XX}\rho(x)>0$ (relevant to the limit $\eps=0$). 
If this is not the case, then we apply criteria (a)(ii)
to the corresponding density. For the lower bound we apply criteria (b) to the density $\rho(x)$.
This completes the proof of Theorem \ref{thm.asym-hyp-rec} for
almost sure bounds on $\min_{j\leq n}X_j(x)$.

\appendix

\section{Technical estimates concerning the invariant density}\label{sec.appendix}
For the proof of Theorem \ref{thm:almost-sure}, case (a)(ii) we can remove the assumption that $\rho$ has to be uniformly H\"older, and allow for a finite number of jump discontinuities in the density,
with left/right limits for $\rho$ as points in $\mathcal{C}$ are approached.
We let  $\mathcal{C}=\{y_1,\ldots,y_N\}$ denote the points of discontinuity, and let $\{L_1,\ldots,L_N\}$ 
denote the corresponding set of jump values, i.e., the differences $|\rho(y_{i}^{+})-\rho(y_{i}^{-})|$. 
We assume that $\rho$ is uniformly H\"older on each $[y_{i},y_{i+1}]$.
We state the following corollary.
\begin{corollary}
Under the assumptions of Theorem \ref{thm:almost-sure},
suppose instead that $\rho$ has a finite number of points of discontinuity as described via
the set $\mathcal{C}$ above. Then statement (a)(ii) remains valid.
\end{corollary}
\begin{proof}
In the proof of Theorem \ref{thm:almost-sure}, the H\"older assumption on $\rho$ is primarily used to obtain an asymptotic on the function $G_r(x)$ in terms of $r$, and in particular on the function 
$e^{-nG_r(x)}$. Suppose now we allow for jumps in the density. Consider $x\in\XX$, and suppose
$\Delta_0(x)=[-r+x,x+r]$ is such that $\mathcal{C}\cap\Delta_0=\emptyset$. Then we obtain the bound
$$|G_r(x)-\rho(x)\cdot 2r|\leq Cr^{1+\eta},$$
using the H\"older assumption on $\rho$ away from $\mathcal{C}$.
Otherwise, suppose that $r$ is small
enough that $\Delta_0$ contains a single point of $\mathcal{C}$, (recall that we take 
$r\equiv r_n\to 0$). Then we obtain
$$|G_r(x)-\rho(x)\cdot 2r|\leq Cr L_i.$$
Considering a worst case scenario where $\rho(x)$ obtains a value of zero on $\Delta_0$, we obtain
$$|e^{-nG_r(x)}-e^{-2nr_n\rho(x)}|\leq |1-e^{-CL_i(\log n)^{\gamma}}|$$
along the sequence $r_n=(\log n)^{\gamma}/2n$. This estimate is trivially bounded by 1.
The measure of the set of $x\in\XX$ such that $\Delta_0(x)\cap\mathcal{C}\neq\emptyset$ is bounded
by $2r\#\mathcal{C}$. Thus, the overall (extra) contribution given to the error term 
as appearing in Lemma \ref{lem:bv-est} is of the order
$$\lesssim \#\mathcal{C} \frac{(\log n)^{\gamma}}{n},$$
and hence of power law type bound in $n$. Thus, the proof of Theorem \ref{thm:almost-sure} goes
through when we allow for a finite number of jumps in the density.
\end{proof}

\begin{remark}
When there is a countable infinite number of jumps in the density then we know that 
$\sum_{i}L_i<\infty$ under the assumption of bounded variation. However
to apply the summation criteria of equation \eqref{eq:asure-sum}
further assumptions are required on the distribution of the jump locations, and their magnitude. 
\end{remark}

\end{document}